\documentclass[a4paper]{article}
\usepackage{graphicx, verbatim, enumerate, amsmath, amsthm, amssymb, amstext, color, amsfonts, hyperref, aliascnt, tikz}
\usepackage[arrow, matrix, curve]{xy}
\usepackage[applemac]{inputenc}

\theoremstyle{plain}
\newtheorem{theo}{Theorem}

\newtheorem{problem}[theo]{Problem}
\newtheorem{lem}[theo]{Lemma}

\newtheorem{cor}[theo]{Corollary}
\newtheorem{Cor}[theo]{Corollary}

\newtheorem{rem}[theo]{Remark}

\theoremstyle{definition}
\newtheorem{defi}[theo]{Definition}
\theoremstyle{definition}

\newcommand{\N}{\ensuremath{\mathbb{N}}}

\newcommand{\sm}{\ensuremath{\smallsetminus}}

\newcommand{\sub}{\subseteq}

\newcommand{\COMMENT}[1]{}

\newcommand{\HF}{\ensuremath{\mathcal H}}

\newcommand{\PF}{\ensuremath{\mathcal P}}

\newcommand{\RF}{\ensuremath{\mathcal R}}
\newcommand{\SF}{\ensuremath{\mathcal S}}
\newcommand{\TF}{\ensuremath{\mathcal T}}

\newenvironment{txteq*}
  {
    \begin{equation*}
    \begin{minipage}[c]{0.85\textwidth} 
    \em                                
  }
  {\end{minipage}\end{equation*}\ignorespacesafterend}

\advance\textheight by -1cm
\begin{document}
\title{Edge-disjoint double rays in infinite graphs:\\ a Halin type result}
\author{
Nathan Bowler\footnote{n.bowler1729@gmail.com. Research supported by the Alexander von Humboldt Foundation.}
\and Johannes Carmesin\footnote{johannes.carmesin@math.uni-hamburg.de. Research supported by the Studienstiftung des deutschen Volkes.}
\and Julian Pott\footnote{Research supported by the Deutsche Forschungsgemeinschaft.}}
\date{Fachbereich Mathematik\\ Universit\"at Hamburg\\ Hamburg, Germany}

\maketitle

\begin{abstract}
We show that any graph that contains $k$ edge-disjoint double rays for any $k\in\N$ contains also infinitely many edge-disjoint double rays.
This was conjectured by Andreae in 1981.
\end{abstract}
\section{Introduction}
We say a graph $G$ has \emph{arbitrarily many vertex-disjoint $H$} if for every $k\in\N$ there is a family of~$k$ vertex-disjoint subgraphs of~$G$ each of which is isomorphic to~$H$.
Halin's Theorem says that every graph that has arbitrarily many vertex-disjoint rays, also has infinitely many vertex-disjoint rays~\cite{halin64}.
In 1970 he extended this result to vertex-disjoint double rays~\cite{halin70}.
Jung proved a strengthening of Halin's Theorem where the initial vertices of the rays are constrained to a certain vertex set~\cite{jung69}.

We look at the same questions with `edge-disjoint' replacing `vertex-disjoint'.
Consider first the statement corresponding to Halin's Theorem.
It suffices to prove this statement in locally finite graphs, as each graph with arbitrarily many edge-disjoint rays contains a locally finite union of tails of these rays.
But the statement for locally finite graphs follows from Halin's original Theorem applied to the line-graph.

This reduction to locally finite graphs does not work for Jung's Theorem or for Halin's statement about double rays.
Andreae proved an analog of Jung's Theorem for edge-disjoint rays in 1981, and conjectured that a Halin-type Theorem would be true for edge-disjoint double rays~\cite{andreae81}.
Our aim in the current paper is to prove this conjecture.

More precisely, we say a graph $G$ has \emph{arbitrarily many edge-disjoint $H$} if for every $k\in\N$ there is a family of~$k$ edge-disjoint subgraphs of~$G$ each of which is isomorphic to~$H$, and our main result is the following.

\begin{theo}\label{main}
Any graph that has arbitrarily many edge-disjoint double rays has infinitely many edge-disjoint double rays.
\end{theo}

Even for locally finite graphs this theorem does not follow from Halin's analogous result for vertex-disjoint double rays applied to the line graph.
For example a double ray in the line graph may correspond, in the original graph, to a configuration as in Figure~\ref{bad_double_ray}.
\begin{figure}[h]
\begin{center}
\includegraphics{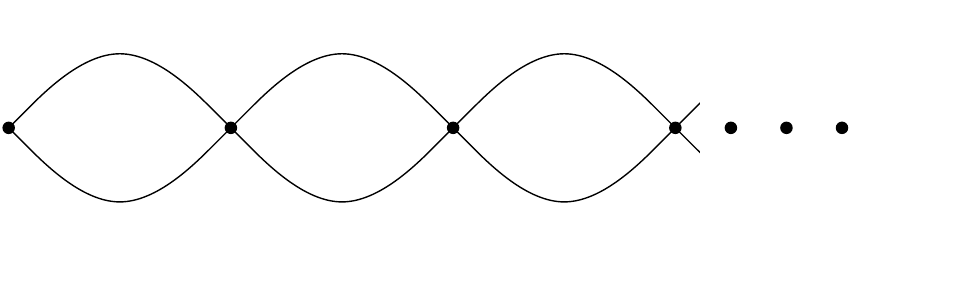}
\end{center}
\caption{A graph that does not include a double ray but whose line graph does.}\label{bad_double_ray}
\end{figure}

A related notion is that of ubiquity.
A graph $H$ is \emph{ubiquitous} with respect to a graph relation $\le$ if $nH\le G$ for all $n\in\N$ implies $\aleph_0 H\le G$, where $nH$ denotes the disjoint union of $n$ copies of~$H$.
For example, Halin's Theorem says that rays are ubiquitous with respect to the subgraph relation.
It is known that not every graph is ubiquitous with respect to the minor relation~\cite{andreae02}, nor is every locally finite graph ubiquitous with respect to the subgraph relation~\cite{lake76, woodall76}, or even the topological minor relation~\cite{andreae02, andreae13}.
However, Andreae has conjectured that every locally finite graph is ubiquitous with respect to the minor relation~\cite{andreae02}.
For more details see~\cite{andreae13}.
In Section~\ref{outlook} (the outlook) we introduce a notion closely related to ubiquity.

The proof is organised as follows.
In Section~\ref{known_cases} we explain how to deal with the cases that the graph has infinitely many ends, or an end with infinite vertex-degree.
In Section~\ref{two_ended} we consider the `two ended' case:
That in which there are two ends $\omega$ and $\omega'$ both of finite vertex-degree, and arbitrarily many edge-disjoint double rays from $\omega$ to $\omega'$.

The only remaining case is the `one ended' case:
That in which there is a single end $\omega$ of finite vertex-degree and arbitrarily many edge-disjoint double rays from $\omega$ to $\omega$.
One central idea in the proof of this case is to consider $2$-rays instead of double rays.
Here a $2$-ray is a pair of vertex-disjoint rays.
For example, from each double ray one obtains a $2$-ray by removing a finite path.
The remainder of the proof is subdivided into two parts:
In Subsection~\ref{arb_to_inf} we show that if there are arbitrarily many edge-disjoint $2$-rays into $\omega$, then there are infinitely many such $2$-rays.
In Subsection~\ref{inf_to_inf} we show that if there are infinitely many edge-disjoint $2$-rays into $\omega$, then there are infinitely many edge-disjoint double rays from $\omega$ to $\omega$.

We finish by discussing the outlook and mentioning some open problems.

\section{Preliminaries}
All our basic notation for graphs is taken from~\cite{DiestelBook10noEE}.
In particular, two rays in a graph are equivalent if no finite set separates them.
The equivalence classes of this relation are called the \emph{ends} of~$G$.
We say that a ray in an end $\omega$ \emph{converges} to~$\omega$.
A double ray \emph{converges} to all the ends of which it includes a ray.

\subsection{The structure of a thin end}
It follows from Halin's Theorem that if there are arbitrarily many vertex-disjoint rays in an end of~$G$, then there are infinitely many such rays.
This fact motivated the central definition of the \emph{vertex-degree} of an end $\omega$:
the maximal cardinality of a set of vertex-disjoint rays in~$\omega$.

An end is \emph{thin} if its vertex-degree is finite, and otherwise it is \emph{thick}.
A pair $(A,B)$ of edge-disjoint subgraphs of~$G$ is a \emph{separation} of~$G$ if $A\cup B= G$.
The number of vertices of $A\cap B$ is called the \emph{order} of the separation.

\begin{defi}\label{captures}
Let $G$ be a locally finite graph and $\omega$ a thin end of~$G$.
A countable infinite sequence $((A_i,B_i))_{i\in\N}$ of separations of~$G$ \emph{captures $\omega$} if for all $i\in\N$
\begin{itemize}
\item $A_i\cap B_{i+1}=\emptyset$,
\item $A_{i+1}\cap B_i$ is connected,
\item $\bigcup_{i\in\N} A_i= G$,
\item the order of $(A_i,B_i)$ is the vertex-degree of $\omega$, and
\item each $B_i$ contains a ray from $\omega$.
\end{itemize}
\end{defi}

\begin{lem}\label{omega-sequence}
Let $G$ be a locally finite graph with a thin end $\omega$. Then there is a sequence that captures $\omega$.
\end{lem}

\begin{proof}
Without loss of generality $G$ is connected, and so is countable.
Let $v_1,v_2,\ldots$ be an enumeration of the vertices of~$G$.
Let $k$ be the vertex-degree of~$\omega$.
Let $\RF=\{R_1,\ldots,R_k\}$ be a set of vertex-disjoint rays in $\omega$ and let $S$ be the set of their start vertices.
We pick a sequence $((A_i,B_i))_{i\in\N}$ of separations and a sequence $(T_i)$ of connected subgraphs recursively as follows.
We pick $(A_i,B_i)$ such that $S$ is included in $A_i$, such that there is a ray from $\omega$ included in $B_i$, and such that $B_i$ does not meet $\bigcup_{j<i}T_j$ or $\{v_j\mid j\le i\}$: subject to this we minimise the size of the set $X_i$ of vertices in $A_i\cap B_i$.
Because of this minimization $B_i$ is connected and $X_i$ is finite.
We take $T_i$ to be a finite connected subgraph of $B_i$ including $X_i$.
Note that any ray that meets all of the $B_i$ must be in~$\omega$.

By Menger's Theorem~\cite{DiestelBook10noEE} we get for each $i\in\N$ a set $\PF_i$ of vertex-disjoint paths from $X_i$ to $X_{i+1}$ of size $|X_i|$.
From these, for each $i$ we get a set of $|X_i|$ vertex-disjoint rays in $\omega$.
Thus the size of $X_i$ is at most $k$.
On the other hand it is at least $k$ as each ray $R_j$ meets each set $X_i$.

Assume for contradiction that there is a vertex $v\in A_i\cap B_{i+1}$. Let $R$ be a ray from $v$ to $\omega$ inside $B_{i+1}$.
Then $R$ must meet $X_i$, contradicting the definition of $B_{i+1}$.
Thus $A_i\cap B_{i+1}$ is empty.

Observe that $\bigcup\PF_i\cup T_i$ is a connected subgraph of $A_{i+1}\cap B_i$ containing all vertices of $X_i$ and $X_{i+1}$.
For any vertex $v\in A_{i+1}\cap B_i$ there is a $v$--$X_{i+1}$ path $P$ in $B_i$.
$P$ meets $B_{i+1}$ only in $X_{i+1}$.
So $P$ is included in $A_{i+1}\cap B_i$.
Thus $A_{i+1}\cap B_i$ is connected.
The remaining conditions are clear.
\end{proof}

\begin{rem}
Every infinite subsequence of a sequence capturing $\omega$ also captures $\omega$.\qed
\end{rem}

The following is obvious:
\begin{rem}\label{lem:finite_path}
Let $G$ be a graph and $v,w\in V(G)$
If $G$ contains arbitrarily many edge-disjoint $v$--$w$ paths, then it contains infinitely many edge-disjoint $v$--$w$ paths.\qed
\end{rem}

We will need the following special case of the theorem of Andreae mentioned in the Introduction.
\begin{theo}[Andreae~\cite{andreae81}]\label{lem:Andreae}
Let $G$ be a graph and $v\in V(G)$.
If there are arbitrarily many edge-disjoint rays all starting at $v$, then there are infinitely many edge-disjoint rays all starting at $v$.
\end{theo}

\section{Known cases}\label{known_cases}
Many special cases of Theorem~\ref{main} are already known or easy to prove.
For example Halin showed the following.
\begin{theo}[Halin]\label{lem:Halin}
Let $G$ be a graph and $\omega$ an end of $G$.
If $\omega$ contains arbitrarily many vertex-disjoint rays, then $G$ has a half-grid as a minor.
\end{theo}

\begin{cor}\label{lem:inf_vx_degree}
Any graph with an end of infinite vertex-degree has infinitely many edge-disjoint double rays.\qed
\end{cor}

Another simple case is the case where the graph has infinitely many ends.

\begin{lem}\label{infinitely_many_ends}
A tree with infinitely many ends contains infinitely many edge-disjoint double rays.
\end{lem}
\begin{proof}
It suffices to show that every tree $T$ with infinitely many ends contains a double ray such that removing its edges leaves a component containing infinitely many ends, since then one can pick those double rays recursively.

There is a vertex $v\in V(T)$ such that $T-v$ has at least $3$ components $C_1,C_2,C_3$ that each have at least one end, as $T$ contains more than $2$ ends.
Let $e_i$ be the edge $vw_i$ with $w_i\in C_i$ for $i\in\{1,2,3\}$.
The graph $T\sm\{e_1,e_2,e_3\}$ has precisely $4$ components ($C_1,C_2,C_3$ and the one containing $v$), one of which, $D$ say, has infinitely many ends.
By symmetry we may assume that $D$ is neither $C_1$ nor $C_2$.
There is a double ray $R$ all whose edges are contained in $C_1\cup C_2 \cup \{e_1, e_2\}$.
Removing the edges of~$R$ leaves the component $D$, which has infinitely many ends.
\end{proof}

\begin{cor}\label{lem:inf_ends}
Any connected graph with infinitely many ends has infinitely many edge-disjoint double rays.\qed
\end{cor}

\section{The `two ended' case}\label{two_ended}
Using the results of Section~\ref{known_cases} it is enough to show that any graph with only finitely many ends, each of which is thin, has infinitely many edge-disjoint double rays as soon as it has arbitrarily many edge-disjoint double rays.
Any double ray in such a graph has to join a pair of ends (not necessarily distinct), and there are only finitely many such pairs.
So if there are arbitrarily many edge-disjoint double rays, then there is a pair of ends such that there are arbitrarily many edge-disjoint double rays joining those two ends.
In this section we deal with the case where these two ends are different, and in Section~\ref{one_end} we deal with the case that they are the same.
We start with two preparatory lemmas.

\begin{lem}\label{simple_ramsey}
Let $G$ be a graph with a thin end $\omega$, and let $\RF\sub\omega$ be an infinite set.
Then there is an infinite subset of $\RF$ such that any two of its members intersect in infinitely many vertices.
\end{lem}
\begin{proof}
We define an auxilliary graph $H$ with $V(H)=\RF$ and an edge between two rays if and only if they intersect in infinitely many vertices.
By Ramsey's Theorem either $H$ contains an infinite clique or an infinite independent set of vertices.
Let us show that there cannot be an infinite independent set in~$H$.
Let $k$ be the vertex-degree of $\omega$: we shall show that $H$ does not have an independent set of size $k+1$.
Suppose for a contradiction that $X\sub \RF$ is a set of $k+1$ rays that is independent in~$H$.
Since any two rays in~$X$ meet in only finitely many vertices, each ray in~$X$ contains a tail that is disjoint to all the other rays in~$X$.
The set of these $k+1$ vertex-disjoint tails witnesses that $\omega$ has vertex-degree at least $k+1$, a contradiction.
Thus there is an infinite clique $K\sub H$, which is the desired infinite subset.
\end{proof}
\begin{lem}\label{choose_your_start}
Let $G$ be a graph consisting of the union of a set $\RF$ of infinitely many edge-disjoint rays of which any pair intersect in infinitely many vertices.
Let $X\sub V(G)$ be an infinite set of vertices, then there are infinitely many edge-disjoint rays in~$G$ all starting in different vertices of~$X$.
\end{lem}
\begin{proof}
If there are infinitely many rays in~$\RF$ each of which contains a different vertex from $X$, then suitable tails of these rays give the desired rays.
Otherwise there is a ray $R\in \RF$ meeting $X$ infinitely often.
In this case, we choose the desired rays recursively such that each contains a tail from some ray in $\RF-R$.
Having chosen finitely many such rays, we can always pick another:
we start at some point in~$X$ on~$R$ which is beyond all the (finitely many) edges on~$R$ used so far.
We follow $R$ until we reach a vertex of some ray $R'$ in $\RF- R$ whose tail has not been used yet, then we follow $R'$.
\end{proof}
\begin{lem}\label{two_ends}\label{lem:two_ends}
Let $G$ be a graph with only finitely many ends, all of which are thin.
Let $\omega_1,\omega_2$ be distinct ends of~$G$.
If $G$ contains arbitrarily many edge-disjoint double rays each of which converges to both $\omega_1$ and $\omega_2$, then $G$ contains infinitely many edge-disjoint double rays each of which converges to both $\omega_1$ and $\omega_2$.
\end{lem}

\begin{proof}
For each pair of ends, there is a finite set separating them. The finite union of these finite sets is a finite set $S\sub V(G)$ separating any two ends of~$G$.
For $i=1,2$ let $C_i$ be the component of $G-S$ containing $\omega_i$.

There are arbitrarily many edge-disjoint double rays from $\omega_1$ to $\omega_2$ that have a common last vertex $v_1$ in $S$ before staying in $C_1$ and also a common last vertex $v_2$ in $S$ before staying in $C_2$. Note that $v_1$ may be equal to $v_2$.
There are arbitrarily many edge-disjoint rays in $C_1+v_1$ all starting in~$v_1$.
By Theorem~\ref{lem:Andreae} there is a countable infinite set $\RF_1=\{R_1^i\mid {i\in\N}\}$ of edge-disjoint rays each included in $C_1+v_1$ and starting in $v_1$.
By replacing $\RF_1$ with an infinite subset of itself, if necessary, we may assume by Lemma~\ref{simple_ramsey} that any two members of $\RF_1$ intersect in infinitely many vertices.
Similarly, there is a countable infinite set $\RF_2=\{R_2^i\mid {i\in\N}\}$ of edge-disjoint rays each included in $C_2+v_2$ and starting in $v_2$ such that any two members of $\RF_2$ intersect in infinitely many vertices.

Let us subdivide all edges in $\bigcup\RF_1$ and call the set of subdivision vertices $X_1$.
Similarly, we subdivide all edges in $\bigcup\RF_2$ and call the set of subdivision vertices $X_2$.
Below we shall find double rays in the subdivided graph, which immediately give rise to the desired double rays in~$G$.

Suppose for a contradiction that there is a finite set $F$ of edges separating $X_1$ from $X_2$.
Then $v_i$ has to be on the same side of that separation as $X_i$ as there are infinitely many $v_i$--$X_i$ edges.
So $F$ separates $v_1$ from $v_2$, which contradicts the fact that there are arbitrarily many edge-disjoint double rays containing both $v_1$ and $v_2$.
By Remark~\ref{lem:finite_path} there is a set $\PF$ of infinitely many edge-disjoint $X_1$--$X_2$ paths.
As all vertices in $X_1$ and $X_2$ have degree $2$, and by taking an infinite subset if necessary, we may assume that each end-vertex of a path in~$\PF$ lies on no other path in~$\PF$.

By Lemma~\ref{choose_your_start} there is an infinite set $Y_1$ of start-vertices of paths in~$\PF$ together with an infinite set $\RF_1'$ of edge-disjoint rays with distinct start-vertices whose set of start-vertices is precisely $Y_1$.
Moreover, we can ensure that each ray in $\RF'_1$ is included in $\bigcup\RF_1$.
Let $Y_2$ be the set of end-vertices in $X_2$ of those paths in $\PF$ that start in~$Y_1$.
Applying Lemma~\ref{choose_your_start} again, we obtain an infinite set $Z_2\sub Y_2$ together with an infinite set $\RF_2'$ of edge-disjoint rays included in $\bigcup \RF_2$ with distinct start-vertices whose set of start-vertices is precisely $Z_2$.

For each path $P$ in~$\PF$ ending in $Z_2$, there is a double ray in the union of $P$ and the two rays from $\RF'_1$ and $\RF_2'$ that $P$ meets in its end-vertices.
By construction, all these infinitely many double rays are edge-disjoint.
Each of those double rays converges to both $\omega_1$ and $\omega_2$, since each $\omega_i$ is the only end in~$C_i$.
\end{proof}

\begin{rem}
Instead of subdividing edges we also could have worked in the line graph of~$G$.
Indeed, there are infinitely many vertex-disjoint paths in the line graph from $\bigcup\RF_1$ to $\bigcup\RF_2$.
\end{rem}

\section{The `one ended' case}\label{one_end}
We are now going to look at graphs $G$ that contain a thin end $\omega$ such that there are arbitrarily many edge-disjoint double rays converging only to the end $\omega$.
The aim of this section is to prove the following lemma, and to deduce Theorem~\ref{main}.

\begin{lem}\label{one-ended}\label{lem:one_ended}
Let $G$ be a countable graph and let $\omega$ be a thin end of~$G$.
Assume there are arbitrarily many edge-disjoint double rays all of whose rays converge to $\omega$.
Then $G$ has infinitely many edge-disjoint double rays.
\end{lem}

We promise that the assumption of countability will not cause problems later.

\subsection{Reduction to the locally finite case}

A key notion for this section is that of a $2$-ray.
A \emph{ $2$-ray} is a pair of vertex-disjoint rays.
For example, from each double ray one obtains a $2$-ray by removing a finite path.

In order to deduce that $G$ has infinitely many edge-disjoint double rays, we will only need that $G$ has arbitrarily many edge-disjoint $2$-rays. In this subsection, we illustrate one advantage of $2$-rays, namely that we may reduce to the case where $G$ is locally finite.

\begin{lem}\label{locally-finite}\label{lem:locally_finite}
Let $G$ be a countable graph with a thin end $\omega$.
Assume there is a countable infinite set $\RF$ of rays all of which converge to $\omega$.

Then there is a locally finite subgraph $H$ of $G$ with a single end which is thin
such that the graph $H$ includes a tail of any $R\in \RF$.
\end{lem}

\begin{proof}
Let $(R_i\mid i\in \N)$ be an enumeration of $\RF$.
Let $(v_i\mid i\in \N)$ be an enumeration of the vertices of $G$.
Let $U_i$ be the unique component of ${G\sm \{v_1,\ldots ,v_i\}}$ including a tail of each ray in $\omega$.

For $i\in \N$, we pick a tail $R_i'$ of $R_i$ in $U_i$.
Let $H_1=\bigcup_{i\in \N} R_i'$.
Making use of $H_1$, we shall construct the desired subgraph $H$. Before that, we shall collect some properties of $H_1$.

As every vertex of $G$ lies in only finitely many of the $U_i$, the graph $H_1$ is locally finite.
Each ray in $H_1$ converges to $\omega$ in $G$
since $H_1\sm U_i$ is finite for every $i\in \N$.
Let $\Psi$ be the set of ends of $H_1$.
Since $\omega$ is thin, $\Psi$ has to be finite: $\Psi=\{\omega_1,\ldots,\omega_n\}$.
For each $i\leq n$, we pick a ray $S_i\subseteq H_1$ converging to $\omega_i$.

Now we are in a position to construct $H$.
For any $i>1$, the
rays $S_1$ and $S_i$ are joined by an infinite set $\PF_{i}$ of vertex-disjoint paths in $G$.
We obtain $H$ from $H_1$ by adding all paths in the sets $\PF_{i}$. Since $H_1$ is locally finite, $H$ is locally finite.

It remains to show that every ray $R$ in $H$ is equivalent to $S_1$.
If $R$ contains infinitely many edges from the $\PF_{i}$, then there is a single
$\PF_{i}$ which $R$ meets infinitely, and thus $R$ is equivalent to $S_1$.
Thus we may assume that a tail of $R$ is a ray in $H_1$. So it converges to some $\omega_i\in \Psi$. Since $S_i$ and $S_1$ are equivalent, $R$ and $S_1$ are equivalent, which completes the proof.
\end{proof}

\begin{cor}\label{locally-finite-2-ray}
Let $G$ be a countable graph with a thin end $\omega$ and arbitrarily many edge-disjoint $2$-rays of which all the constituent rays converge to $\omega$. Then there is a locally finite subgraph $H$ of $G$ with a single end, which is thin, such that $H$ has arbitrarily many edge-disjoint $2$-rays.
\end{cor}
\begin{proof}
By Lemma~\ref{lem:locally_finite} there is a locally finite graph $H\sub G$ with a single end such that a tail of each of the constituent rays of the arbitrarily many $2$-rays is included in~$H$.
\end{proof}

\subsection{Double rays versus $2$-rays}\label{inf_to_inf}

A connected subgraph of a graph $G$ including a vertex set $S\sub V(G)$ is a \emph{connector} of~$S$ in $G$.

\begin{lem}\label{finite_connector}
Let $G$ be a connected graph and $S$ a finite set of vertices of~$G$.
Let $\HF$ be a set of edge-disjoint subgraphs $H$ of~$G$ such that each connected component of~$H$ meets $S$.
Then there is a finite connector $T$ of~$S$, such that at most $2|S|-2$ graphs from $\HF$ contain edges of~$T$.
\end{lem}

\begin{proof}
By replacing $\HF$ with the set of connected components of graphs in $\HF$, if necessary, we may assume that each member of $\HF$ is connected.
We construct graphs $T_i$ recursively for $0\le i< |S|$ such that each $T_i$ is finite and has at most $|S|-i$ components, at most $2i$ graphs from $\HF$ contain edges of~$T_i$, and each component of $T_i$ meets $S$.
Let $T_0= (S,\emptyset)$ be the graph with vertex set $S$ and no edges.
Assume that $T_i$ has been defined.

If $T_{i}$ is connected let $T_{i+1}=T_i$.
For a component $C$ of $T_i$, let $C'$ be the graph obtained from $C$ by adding all graphs from $\HF$ that meet $C$.

As $G$ is connected, there is a path $P$ (possibly trivial) in~$G$ joining two of these subgraphs $C_1'$ and $C_2'$ say.
And by taking the length of~$P$ minimal, we may assume that $P$ does not contain any edge from any $H\in\HF$.
Then we can extend $P$ to a $C_1$--$C_2$ path $Q$ by adding edges from at most two subgraphs from $\HF$ --- one included in $C_1'$ and the other in $C_2'$.
We obtain $T_{i+1}$ from $T_i$ by adding $Q$.

$T=T_{|S|-1}$ has at most one component and thus is connected.
And at most $2|S|-2$ many graphs from $\HF$ contain edges of~$T$.
Thus $T$ is as desired.
\end{proof}

Let $d,d'$ be $2$-rays.
$d$ is a \emph{tail} of $d'$ if each ray of~$d$ is a tail of a ray of~$d'$.
A set $D'$ is a \emph{tailor} of a set $D$ of $2$-rays if each element of $D'$ is a tail of some element of $D$ but no $2$-ray in $D$ includes more than one $2$-ray in $D'$.

\begin{lem}\label{infinitely_connectors}
Let $G$ be a locally finite graph with a single end $\omega$, which is thin. Assume that $G$ contains an infinite set $D=\{d_1,d_2,\dots\}$ of edge-disjoint $2$-rays.

Then $G$ contains an infinite tailor $D'$ of~$D$ and a sequence $((A_i,B_i))_{i\in\N}$ capturing $\omega$ (see Definition~\ref{captures}) such that there is a family of vertex-disjoint connectors $T_i$ of $A_i\cap B_i$ contained in $A_{i+1}\cap B_i$, each of which is edge-disjoint from each member of $D'$.
\end{lem}

\begin{proof}
Let $k$ be the vertex-degree of $\omega$.
By Lemma~\ref{omega-sequence} there is a sequence $((A_i',B'_i))_{i\in\N}$ capturing $\omega$.
By replacing each $2$-ray in $D$ with a tail of itself if necessary, we may assume that for all $(r,s)\in D$ and $i\in\N$ either both $r$ and $s$ meet $A'_i$ or none meets $A'_i$.
By Lemma~\ref{finite_connector} there is a finite connector $T'_i$ of $A'_i\cap B'_i$ in the connected graph $B'_i$ which meets in an edge at most $2k-2$ of the $2$-rays of~$D$ that have a vertex in $A'_i$.

Thus, there are at most $2k-2$ $2$-rays in~$D$ that meet all but finitely many of the $T_i'$ in an edge.
By throwing away these finitely many $2$-rays in~$D$ we may assume that each $2$-ray in~$D$ is edge-disjoint from infinitely many of the $T'_i$.
So we can recursively build a sequence $N_1,N_2,\dots$ of infinite sets of natural numbers such that $N_i\supseteq N_{i+1}$, the first $i$ elements of $N_i$ are all contained in $N_{i+1}$, and $d_i$ only meets finitely many of the $T'_j$ with $j\in N_i$ in an edge.
Then $N=\bigcap_{i\in\N} N_i$ is infinite and has the property that each $d_i$ only meets finitely many of the $T'_j$ with $j\in N$ in an edge.
Thus there is an infinite tailor $D'$ of $D$ such that no $2$-ray from $D'$ meets any $T'_j$ for $j\in N$ in an edge.

We recursively define a sequence $n_1,n_2,\dots$ of natural numbers by taking $n_i\in N$ sufficiently large that $B'_{n_i}$ does not meet $T'_{n_j}$ for any $j<i$
.
Taking $(A_i,B_i)=(A'_{n_i},B'_{n_i})$ and $T_i= T'_{n_i}$ gives the desired sequences.
\end{proof}

\begin{lem}\label{lem_2-rays_to_double_rays}\label{lem:2-rays_suffice}
If a locally finite graph $G$ with a single end $\omega$ which is thin contains infinitely many edge-disjoint $2$-rays, then $G$ contains infinitely many edge-disjoint double rays.
\end{lem}

\begin{proof}
Applying Lemma~\ref{infinitely_connectors} we get an infinite set $D$ of edge-disjoint $2$-rays, a sequence $((A_i,B_i))_{i\in\N}$ capturing $\omega$, and connectors $T_i$ of $A_i\cap B_i$  for each $i\in\N$ such that the $T_i$ are vertex-disjoint from each other and edge-disjoint from all members of~$D$.

We shall construct the desired set of infinitely many edge-disjoint double rays as a nested union of sets $D_i$. We construct the $D_i$ recursively.
Assume that a set $D_i$ of $i$ edge-disjoint double rays has been defined such that each of its members is included in the union of a single $2$-ray from $D$ and one connector $T_j$.
Let $d_{i+1}\in D$ be a $2$-ray distinct from the finitely many $2$-rays used so far.
Let $C_{i+1}$ be one of the infinitely many connectors that is different from all the finitely many connectors used so far and that meets both rays of $d_{i+1}$.
Clearly, $d_{i+1}\cup C_{i+1}$ includes a double ray $R_{i+1}$.
Let $D_{i+1}=D_i\cup\{R_{i+1}\}$.
The union $\bigcup_{i\in\N} D_i$ is an infinite set of edge-disjoint double rays as desired.
\end{proof}

\subsection{Shapes and allowed shapes}\label{arb_to_inf}
Let $G$ be a graph and $(A,B)$ a separation of~$G$.
A \emph{shape} for $(A,B)$ is a word $v_1x_1v_2x_2\dots x_{n-1}v_n$ with $v_i\in A\cap B$ and $x_i\in\{l,r\}$ such that no vertex appears twice.
We call the $v_i$ the \emph{vertices} of the shape.
Every ray $R$ \emph{induces} a shape $\sigma=\sigma_R(A,B)$ on every separation $(A,B)$ of finite order in the following way:
Let $<_R$ be the \emph{natural order} on $V(R)$ induced by the ray, where $v<_R w$ if $w$ lies in the unique infinite component of $R-v$.
The vertices of~$\sigma$ are those vertices of~$R$ that lie in~$A\cap B$ and they appear in $\sigma$ in the order given by $<_R$.
For $v_i,v_{i+1}$ the path $v_i R v_{i+1}$ has edges only in $A$ or only in $B$ but not in both.
In the first case we put $l$ between $v_i$ and $v_{i+1}$ and in the second case we put $r$ between $v_i$ and $v_{i+1}$.

Let $(A_1,B_1),(A_2,B_2)$ be separations with $A_1\cap B_2=\emptyset$ and thus also $A_1\sub A_2$ and $B_2\sub B_1$.
Let $\sigma_i$ be a nonempty shape for $(A_i,B_i)$.
The word $\tau = v_1x_1v_2\dots x_{n-1}v_n$ is an \emph{allowed shape linking $\sigma_1$ to $\sigma_2$} with {\em vertices} $v_1 \ldots v_n$ if the following holds.
\begin{itemize}
\item $v$ is a vertex of $\tau$ if and only if it is a vertex of $\sigma_1$ or $\sigma_2$,
\item if $v$ appears before $w$ in $\sigma_i$, then $v$ appears before $w$ in $\tau$,
\item $v_1$ is the initial vertex of $\sigma_1$ and $v_n$ is the terminal vertex of $\sigma_2$,
\item $x_i\in\{l,m,r\}$,
\item the subword $vlw$ appears in $\tau$ if and only if it appears in $\sigma_1$,
\item the subword $vrw$ appears in $\tau$ if and only if it appears in $\sigma_2$,
\item $v_i\neq v_j$ for $i\neq j$.
\end{itemize}

Each ray $R$ defines a word $\tau=\tau_R[(A_1,B_1),(A_2,B_2)]= v_1x_1v_2\ldots x_{n-1}v_{n}$ with vertices $v_i$ and $x_i\in\{l,m,r\}$ as follows.
The vertices of~$\tau$ are those vertices of~$R$ that lie in~$A_1\cap B_1$ or $A_2\cap B_2$ and they appear in $\tau$ in the order given by $<_R$.
For $v_i,v_{i+1}$ the path $v_i R v_{i+1}$ has edges either only in $A_1$, only in $A_2\cap B_1$, or only in $B_2$.
In the first case we set $x_i=l$ and $\tau$ contains the subword $v_ilv_{i+1}$.
In the second case we set $x_i=m$ and $\tau$ contains the subword $v_imv_{i+1}$.
In the third case we set $x_i=r$ and $\tau$ contains the subword $v_irv_{i+1}$.

For a ray $R$ to induce an allowed shape $\tau_R[(A_1,B_1),(A_2,B_2)]$ we need at least that $R$ starts in~$A_2$.
However, each ray in~$\omega$ has a tail such that whenever it meets an $A_i$ it also starts in that $A_i$.
Let us call such rays \emph{lefty}. A $2$-ray is \emph{lefty} if both its rays are.

\begin{rem}
Let $(A_1,B_1)$, and $(A_2,B_2)$ be two separations of finite order with $A_1\subseteq A_2$, and $B_2\subseteq B_1$. For every lefty ray $R$ meeting $A_1$, the word $\tau_R[(A_1,B_1),(A_2,B_2)]$ is an allowed shape linking $\sigma_R(A_1,B_1)$ and $\sigma_R(A_2,B_2)$.\qed
\end{rem}

From now on let us fix a locally finite graph $G$ with a thin end $\omega$ of vertex-degree $k$.
And let $((A_i,B_i))_{i\in\N}$ be a sequence capturing $\omega$ such that each member has order $k$.

A \emph{$2$-shape} for a separation $(A,B)$ is a pair of shapes for $(A,B)$.
Every $2$-ray induces a $2$-shape coordinatewise in the obvious way.
Similarly, an \emph{allowed $2$-shape} is a pair of allowed shapes.

Clearly, there is a global constant $c_1\in\N$ depending only on $k$ such that there are at most $c_1$ distinct $2$-shapes for each separation $(A_i,B_i)$.
Similarly, there is a global constant $c_2\in\N$ depending only on $k$ such that for all $i,j\in\N$ there are at most $c_2$ distinct allowed $2$-shapes linking a $2$-shape for $(A_i,B_i)$ with a $2$-shape for $(A_j,B_j)$.

For most of the remainder of this subsection we assume that for every $i\in\N$ there is a set $D_i$ consisting of at least $c_1\cdot c_2\cdot i$ edge-disjoint $2$-rays in~$G$. Our aim will be to show that in these circumstances there must be infinitely many edge-disjoint 2-rays.

By taking a tailor if necessary, we may assume that every $2$-ray in each $D_i$ is lefty.

\begin{lem}\label{same_shape_internal}
There is an infinite set $J\sub \N$ and, for each $i\in\N$, a tailor $D_i'$ of $D_i$ of cardinality $c_2\cdot i$ such that for all $i\in\N$ and $j\in J$ all $2$-rays in $D'_i$ induce the same $2$-shape $\sigma[i,j]$ on $(A_j,B_j)$.
\end{lem}

\begin{proof}
We recursively build infinite sets $J_i\sub\N$ and tailors $D_i'$ of $D_i$ such that for all $k\le i$ and $j\in J_i$ all $2$-rays in $D'_k$ induce the same $2$-shape on $(A_j,B_j)$.
For all $i\ge 1$, we shall ensure that $J_i$ is an infinite subset of $J_{i-1}$ and that the $i-1$ smallest members of $J_i$ and $J_{i-1}$ are the same. We shall take $J$ to be the intersection of all the $J_i$.

Let $J_0=\N$ and let $D_0'$ be the empty set.
Now, for some $i\ge 1$, assume that sets $J_k$ and $D_k'$ have been defined for all $k<i$.
By replacing $2$-rays in $D_i$ by their tails, if necessary, we may assume that each $2$-ray in $D_i$ avoids $A_\ell$, where $\ell$ is the $(i-1)$st smallest value of $J_{i-1}$.
As $D_i$ contains $c_1\cdot c_2\cdot i$ many $2$-rays, for each $j\in J_{i-1}$ there is a set $S_j\sub D_i$ of size at least $c_2\cdot i$ such that each $2$-ray in $S_j$ induces the same $2$-shape on $(A_j,B_j)$.
As there are only finitely many possible choices for $S_j$, there is an infinite subset $J_i$ of $J_{i-1}$ on which $S_j$ is constant.
For $D_i'$ we pick this value of $S_j$.
Since each $d\in D_i'$ induces the empty $2$-shape on each $(A_k,B_k)$ with $k\le \ell$ we may assume that the first $i-1$ elements of $J_{i-1}$ are also included in $J_i$.

It is immediate that the set $J=\bigcap_{i\in\N} J_i$ and the $D_i'$ have the desired property.
\end{proof}

\begin{lem}\label{same_shape_external}
There are two strictly increasing sequences $(n_i)_{i\in\N}$ and $(j_i)_{i\in\N}$ with $n_i\in\N$ and $j_i\in J$ for all $i\in\N$ such that $\sigma[n_i,j_i]=\sigma[n_{i+1},j_i]$ and $\sigma[n_{i},j_i]$ is not empty.
\end{lem}

\begin{proof}
Let $H$ be the graph on $\N$ with an edge $vw\in E(H)$ if and only if there are infinitely many elements $j\in J$ such that $\sigma[v,j]=\sigma[w,j]$.

As there are at most $c_1$ distinct $2$-shapes for any separator $(A_i,B_i)$, there is no independent set of size $c_1+1$ in~$H$ and thus no infinite one.
Thus, by Ramsey's theorem, there is an infinite clique in~$H$.
We may assume without loss of generality that $H$ itself is a clique by moving to a subsequence of the $D'_i$ if necessary.
With this assumption we simply pick $n_i=i$.

Now we pick the $j_i$ recursively.
Assume that $j_i$ has been chosen.
As $i$ and $i+1$ are adjacent in~$H$, there are infinitely many indicies $\ell\in\N$ such that $\sigma[i,\ell]=\sigma[i+1,\ell]$.
In particular, there is such an $\ell> j_i$ such that $\sigma[i+1,\ell]$ is not empty.
We pick $j_{i+1}$ to be one of those $\ell$.

Clearly, $(j_i)_{i\in\N}$ is an increasing sequence and $\sigma[i,j_i]=\sigma[{i+1},j_i]$ as well as $\sigma[{i},j_i]$ is non-empty for all $i\in\N$, which completes the proof.
\end{proof}

By moving to a subsequence of $(D_i')$ and $((A_j,B_j))$, if necessary, we may assume by Lemma~\ref{same_shape_internal} and Lemma~\ref{same_shape_external} that for all $i,j\in\N$ all $d\in D_i'$ induce the same $2$-shape $\sigma[i,j]$ on $(A_j,B_j)$, and that $\sigma[i,i]=\sigma[{i+1},i]$, and that $\sigma[i,i]$ is non-empty.

\begin{lem}\label{same_allowed_shape}
For all $i\in\N$ there is $D_i''\sub D_i'$ such that $|D_i''|= i$, and all $d\in D_i''$ induce the same allowed $2$-shape $\tau[i]$ that links $\sigma[i,i]$ and $\sigma[i,i+1]$.
\end{lem}

\begin{proof}
Note that it is in this proof that we need all the $2$-rays in $D_i''$ to be lefty as they need to induce an allowed $2$-shape that links $\sigma[i,i]$ and $\sigma[i,i+1]$ as soon as they contain a vertex from $A_i$.
As $|D_i'|\ge i\cdot c_2 $ and as there are at most $c_2$ many distinct allowed $2$-shapes that link $\sigma[i,i]$ and $\sigma[i,i+1]$ there is $D_i''\sub D_i'$ with $|D_i''|= i$ such that all $d\in D_i''$ induce the same allowed $2$-shape.
\end{proof}

We enumerate the elements of $D_j''$ as follows: $d^j_1, d^j_2, \dots, d^j_j$.
Let $(s_i^j,t_i^j)$ be a representation of $d_i^j$.
Let $S_i^j=s^j_i\cap A_{j+1}\cap B_j$, and let $\SF_i=\bigcup_{j\ge i} S_i^j$.
Similarly, let $T_i^j=t_i^j\cap A_{j+1}\cap B_j$, and let $\TF_i=\bigcup_{j\ge i} T_i^j$.

Clearly, $\SF_i$ and $\TF_i$ are vertex-disjoint and any two graphs in $\bigcup_{i\in\N}\{\SF_i,\TF_i\}$ are edge-disjoint.
We shall find a ray $R_i$ in each of the $\SF_i$ and a ray $R_i'$ in each of the $\TF_i$.
The infinitely many pairs $(R_i,R_i')$ will then be edge-disjoint $2$-rays, as desired.

\begin{lem}\label{degree_2}
Each vertex $v$ of $\SF_i$ has degree at most $2$.
If $v$ has degree $1$ it is contained in $A_i\cap B_i$.
\end{lem}
\begin{proof}
Clearly, each vertex $v$ of $\SF_i$ that does not lie in any separator $A_j\cap B_j$ has degree $2$, as it is contained in precisely one $S_i^j$, and all the leaves of $S_i^j$ lie in $A_j\cap B_j$ and $A_{j+1}\cap B_{j+1}$ as $d_i^j$ is lefty.
Indeed, in $S_i^j$ it is an inner vertex of a path and thus has degree $2$ in there.
If $v$ lies in $A_i\cap B_i$ it has degree at most $2$, as it is only a vertex of $S_i^j$ for one value of $j$, namely $j = i$.

Hence, we may assume that $v\in A_j\cap B_j$ for some $j>i$.
Thus, $\sigma[j,j]$ contains $v$ and $l:\sigma[j,j]:r$ contains precisely one of the four following subwords:
$$
lvl, lvr, rvl, rvr
$$
(Here we use the notation $p:q$ to denote the concatenation of the word $p$ with the word $q$.)
In the first case $\tau[j-1]$ contains $mvm$ as a subword and $\tau[j]$ has no $m$ adjacent to $v$.
Then $S_i^{j-1}$ contains precisely $2$ edges adjacent to $v$ and $S_i^j$ has no such edge.
The fourth case is the first one with $l$ and $r$ and $j$ and $j-1$ interchanged.

In the second and third cases,  each of $\tau[j-1]$ and $\tau[j]$ has precisely one $m$ adjacent to $v$.
So both $S_i^{j-1}$ and $S_i^j$ contain precisely $1$ edge adjacent to $v$.

As $v$ appears only as a vertex of $S_i^\ell$ for $\ell = j$ or $\ell = j-1$, the degree of $v$ in $\SF_i$ is~$2$.
\end{proof}

\begin{lem}\label{odd_vertices}
There are an odd number of vertices in $\SF_i$ of degree $1$.
\end{lem}
\begin{proof}
By Lemma~\ref{degree_2} we have that each vertex of degree $1$ lies in $A_i\cap B_i$.
Let $v$ be a vertex in $A_i\cap B_i$.
Then, $\sigma[i,i]$ contains $v$ and $l:\sigma[i,i]:r$ contains precisely one of the four following subwords:
$$
lvl, lvr, rvl, rvr
$$
In the first and fourth case $v$ has even degree.
It has degree $1$ otherwise.
As $l:\sigma[i,i]:r$ starts with $l$ and ends with $r$, the word $lvr$ appear precisely once more than the word $rvl$.
Indeed, between two occurrences of $lvr$ there must be one of $rvl$ and vice versa.
Thus, there are an odd number of vertices with degree $1$ in~$\SF_i$.
\end{proof}

\begin{lem}\label{includes_ray}
$\SF_i$ includes a ray.
\end{lem}

\begin{proof}
By Lemma~\ref{degree_2} every vertex of~$\SF_i$ has degree at most $2$ and thus every component of $\SF_i$ has at most two vertices of degree $1$.
By Lemma~\ref{odd_vertices} $\SF_i$ has a component $C$ that contains an odd number of vertices with degree $1$.
Thus $C$ has precisely one vertex of degree $1$ and all its other vertices have degree $2$, thus $C$ is a ray.
\end{proof}

\begin{Cor}\label{lem:inf_2-rays}
$G$ contains infinitely many edge-disjoint $2$-rays.
\end{Cor}
\begin{proof}
By symmetry, Lemma~\ref{includes_ray} is also true with $\TF_i$ in place of $\SF_i$.
Thus $\SF_i\cup\TF_i$ includes a $2$-ray $X_i$.
The $X_i$ are edge-disjoint by construction.
\end{proof}

Recall that Lemma~\ref{lem:one_ended} states that a countable graph with a thin end $\omega$ and arbitrarily many edge-disjoint double rays all whose subrays converge to $\omega$, also has infinitely many edge-disjoint double rays.
We are now in a position to prove this lemma.

\begin{proof}[Proof of Lemma~\ref{lem:one_ended}]
By Lemma~\ref{lem:2-rays_suffice} it suffices to show that~$G$ contains a subgraph $H$ with a single end which is thin such that $H$ has infinitely many edge-disjoint $2$-rays.
By Corollary~\ref{locally-finite-2-ray}, $G$ has a subgraph $H$ with a single end which is thin such that $H$ has arbitrarily many edge-disjoint $2$-rays.
But then by the argument above $H$ contains infinitely many edge-disjoint $2$-rays, as required.
\end{proof}

With these tools at hand, the remaining proof of Theorem~\ref{main} is easy.
Let us collect the results proved so far to show that each graph with arbitrarily many edge-disjoint double rays also has infinitely many edge-disjoint double rays.

\begin{proof}[Proof of Theorem~\ref{main}]
Let $G$ be a graph that has a set $D_i$ of~$i$ edge-disjoint double rays for each $i\in\N$.
Clearly, $G$ has infinitely many edge-disjoint double rays if its subgraph $\bigcup_{i\in\N}D_i$ does, and thus we may assume without loss of generality that $G=\bigcup_{i\in\N}D_i$.
In particular, $G$ is countable.

By Corollary~\ref{lem:inf_ends} we may assume that each connected component of~$G$ includes only finitely many ends.
As each component includes a double ray we may assume that $G$ has only finitely many components.
Thus, there is one component containing arbitrarily many edge-disjoint double rays, and thus we may assume that $G$ is connected.

By Corollary~\ref{lem:inf_vx_degree} we may assume that all ends of~$G$ are thin.
Thus, as mentioned at the start of Section~\ref{two_ended}, there is a pair of ends $(\omega,\omega')$ of $G$ (not necessarily distinct) such that $G$ contains arbitrarily many edge-disjoint double rays each of which converges precisely to $\omega$ and $\omega'$.
This completes the proof as, by Lemma~\ref{lem:two_ends} $G$ has infinitely many edge-disjoint double rays if $\omega$ and $\omega'$ are distinct and by Lemma~\ref{lem:one_ended} $G$ has infinitely many edge-disjoint double rays if $\omega=\omega'$.
\end{proof}

\section{Outlook and open problems}\label{outlook}
We will say that a graph $H$ is \emph{edge-ubiquitous} if every graph having arbitrarily many edge-disjoint $H$ also has infinitely many edge-disjoint $H$.

Thus Theorem~\ref{main} can be stated as follows:
the double ray is edge-ubiquitous.
Andreae's Theorem implies that the ray is edge-ubiquitous.
And clearly, every finite graph is edge-ubiquitous.

We could ask which other graphs are edge-ubiquitous.
It follows from our result that the $2$-ray is edge-ubiquitous.
Let $G$ be a graph in which there are arbitrarily many edge-disjoint $2$-rays.
Let $v*G$ be the graph obtained from $G$ by adding a vertex $v$ adjacent to all vertices of~$G$.
Then $v*G$ has arbitrarily many edge-disjoint double rays, and thus infinitely many
edge-disjoint double rays.
Each of these double rays uses $v$ at most once and thus includes a $2$-ray of~$G$.

The vertex-disjoint union of $k$ rays is called a \emph{$k$-ray}.
The $k$-ray is edge-ubiquitous.
This can be proved with an argument similar to that for Theorem~\ref{main}:
Let $G$ be a graph with arbitrarily many  edge-disjoint $k$-rays.
The same argument as in Corollaries~\ref{lem:inf_ends} and~\ref{lem:inf_vx_degree} shows that we may assume that $G$ has only finitely many ends, each of which is thin.
By removing a finite set of vertices if necessary we may assume that each component of~$G$
has at most one end, which is thin.
Now we can find numbers $k_C$ indexed by the components $C$ of $G$ and summing to $k$ such that each component $C$ has arbitrarily many edge-disjoint $k_C$-rays.
Hence, we may assume that $G$ has only a single end, which is thin.
By Lemma~\ref{lem:locally_finite} we may assume that $G$ is locally finite.

In this case, we use an argument as in Subsection~\ref{arb_to_inf}. It is necessary to use $k$-shapes instead of $2$-shapes but other than that we can use the same combinatorial principle. If $C_1$ and $C_2$ are finite sets, a {\em $(C_1, C_2)$-shaping} is a pair $(c_1, c_2)$ where $c_1$ is a partial colouring of $\N$ with colours from $C_1$ which is defined at all but finitely many numbers and $c_2$ is a colouring of $\N^{(2)}$ with colours from $C_2$ (in our argument above, $C_1$ would be the set of all $k$-shapes and $C_2$ would be the set of all allowed $k$-shapes for all pairs of $k$-shapes).

\begin{lem}\label{lem:principal}
 Let $D_1, D_2, \ldots$ be a sequence of sets of $(C_1, C_2)$-shapings where $D_i$ has size $i$. Then there are strictly increasing sequences $i_1, i_2, \ldots$ and $j_1, j_2, \ldots$ and subsets $S_n \subseteq D_{i_n}$ with $|S_n| \geq n$ such that
\begin{itemize}
 \item for any $n \in \N$ all the values of $c_1(j_n)$ for the shapings $(c_1, c_2) \in S_{n-1} \cup S_n$ are equal (in particular, they are all defined).
\item for any $n \in N$, all the values of $c_2(j_n, j_{n+1})$ for the shapings $(c_1, c_2) \in S_n$ are equal.
\end{itemize}

\end{lem}

Lemma \ref{lem:principal} can be proved by the same method with which we constructed the sets $D_i''$ from the sets $D_i$. The advantage of Lemma~\ref{lem:principal} is that it can not only be applied to $2$-rays but also to more complicated graphs like $k$-rays.

A \emph{talon} is a tree with a single vertex of degree $3$ where all the other vertices have degree $2$. An argument as in Subsection~\ref{inf_to_inf} can be used to deduce that talons are edge-ubiquitous from the fact that $3$-rays are.
However, we do not know whether the graph in Figure~\ref{frog} is edge-ubiquitous.

\begin{figure}[h]
\begin{center}
\includegraphics{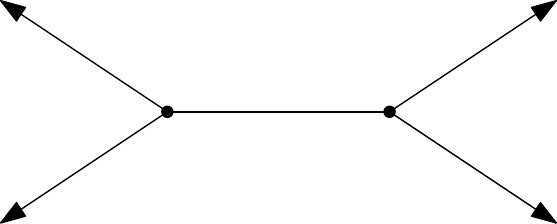}
\end{center}
\caption{A graph obtained from $2$ disjoint double rays, joined by a single edge. Is this graph edge-ubiquitous?}\label{frog}
\end{figure}

We finish with the following open problem.

\begin{problem}
 Is the directed analogue of Theorem~\ref{main} true? More precisely: Is it true that if a directed graph has arbitrarily many edge-disjoint directed double rays, then it has
infinitely many edge-disjoint directed double rays?
\end{problem}

It should be noted that if true the directed analogue would be a common generalization
of Theorem~\ref{main} and the fact that double rays are ubiquitous with respect to the subgraph relation.

\section{Acknowledgement}
We appreciate the helpful and accurate comments of a referee.
\bibliographystyle{plain}
\bibliography{collective}
\end{document}